\documentclass[submission,copyright,creativecommons]{eptcs}
\providecommand{\event}{AiML 2026}

\usepackage{iftex}
\usepackage{stmaryrd}
\newcommand{\noopsort}[1]{}

\ifpdf
  \usepackage{underscore}         
  \usepackage[T1]{fontenc}        
\else
  \usepackage{breakurl}           
\fi

\usepackage{graphicx} 
\usepackage{amsmath}
\usepackage{amsfonts}
\usepackage{yfonts}
\usepackage{amsthm}
\usepackage{amssymb}
\usepackage{stmaryrd}
\usepackage{proof}
\usepackage{mathtools}
\usepackage{modalops}
\usepackage{bussproofs-extra}
\usepackage{mathrsfs}
\usepackage{multicol}
\usepackage{subcaption}
\usepackage{makecell}

\usepackage{xargs} 

\usepackage{aiml26}

\title{Ultrafilter Extensions for Veltman Semantics}
\author{Fèlix {Frigola González}
\institute{University of Barcelona}
\and
Joost J.~Joosten
\institute{University of Barcelona}
\and Vicent {Navarro Arroyo}
\institute{Polytechnic University of València}
\institute{University of Barcelona}
\and Cosimo {Perini Brogi}
\institute{IMT School for Advanced Studies Lucca}
}

\newcommand{\titlerunning}{Ultrafilter Extensions for Veltman Semantics}
\newcommand{\authorrunning}{Frigola González, Joosten, Navarro Arroyo, and Perini Brogi}
\newcommand{\copyrightholders}{Frigola González, Joosten,\\ Navarro Arroyo, and Perini Brogi}
\hypersetup{
  bookmarksnumbered,
  pdftitle    = {\titlerunning},
  pdfauthor   = {\authorrunning},
  pdfsubject  = {EPTCS},               
  pdfkeywords = {keyword1, keyword2} 
}

\begin{document}
\maketitle
\begin{abstract}
In this paper, we present a first-order frame condition for interpretability logic and show that the condition is not modally definable. Yet, the frame condition holds both on ILM and on ILP frames and, hence, is of potential importance for the long-standing open problem about the interpretability logic of all reasonable arithmetical theories. In the light of the Goldblatt-Thomason Theorem, the modally inexpressible frame condition serves as motivation to develop ultrafilter extensions for interpretability logic. We develop the necessary algebraic tools to define these ultrafilter extensions and prove the main properties about both the tools and the ultrafilter extensions.
\end{abstract} 

\section{Introduction}
Building upon early work by Montagna and \v{S}vejdar \cite{svejdar1983, montagna1987}, interpretability logics were introduced by Visser in the 1980s as a natural extension of provability logics (\cite{Visser:1990:InterpretabilityLogic}). Since then, the field of modal interpretability logics has seen a steady process of becoming a mature discipline. A Kripke style semantics for interpretability logics goes by the name of Veltman Semantics. After some early modal completeness and decidability results in \cite{JonghVeltman:1990:ProvabilityLogicsForRelativeInterpretability} it was seen that many interpretability logics satisfy various other good properties too, like the Finite Model Property, the Fixpoint Theorem, or Interpolation (see \cite{Visser:1997:OverviewIL,JonghJaparidze:1998:HandbookPTProvabilityLogic}). 

Proof calculi are mostly Hilbert style but sequent style systems and cut-elimination have been studied in \cite{HakoniemiJoosten:2016:TableauxForInterpretabilityLogics, Sasaki:2002:CutFreeIL, DBLP:journals/rsl/BrogiNO25}. Topological semantics for interpretability logics have been studied in \cite{Iwata:2021:TopologicalSemantics}. Verbrugge semantics (\cite{JoostenMasMikecVukovic:2024:OverviewVerbrugge}) is a generalization of Veltman semantics that is reminiscent of neighbourhood semantics. However, plain Veltman semantics is still the predominant semantics used for interpretability logics. 


Unary normal modal logics are all extensions of the basic modal logic \logicK. The basic modal interpretability logic is called \il although on occasion logics below \il are studied too (\cite{Kurahashi:2021:ModalCompleteness,Okawa:2024:UnaryInterpretabilityLogics}). Various extensions of \il by axiom schemes induce correspondences to first- or higher-order properties of frames. 

To the best of our knowledge, in this paper -- Section \ref{section:NonDefinability}-- we present the first frame condition (which is first-order) for interpretability logics with a proof that the condition is not modally definable. Moreover, the frame condition is natural in light of the semantically defined interpretability logics 
$\ilal$ 
and $\ilalfrak$ 
(see \cite{JoostenVisser:2000:IntLogicAll, Joosten:2015:TwoSeries, NavarroJoosten:2025:tamesemantics}).

Our proof of modal undefinability is rather ad-hoc and this begs the question if a Goldblatt-Thomason Theorem for interpretability logic can be proven to deal with undefinability results in a more uniform way. In the remainder of the paper we develop a first natural step toward such a Goldblatt-Thomason theorem.  In Section \ref{section:ultrafilterExtensions} we define ultrafilter extensions and prove some elementary properties of them. More sophisticated arguments require that we develop labeling techniques in the algebraic setting which is done in Section \ref{section:fullLabelsAlgebras}.
The main results on ultrafilter extensions are presented in Section \ref{section:mainResults}: elementary equivalence and modal saturation. Further steps towards the Goldblatt-Thomason Theorem will be addressed in future work. We start with some preliminaries in the next section. 

\section{Preliminaries}

Following \cite{Visser:1988:preliminaryNotesOnInterpretabilityLogic,Visser:1990:InterpretabilityLogic}, throughout this paper, we will work with a multimodal language \lang, whose formulas ($\alpha,\beta,\ldots\in\form$) are built on top of a given denumerably infinite set of propositional atoms ($p,q,\ldots\in\atom$) according to the following grammar:
    $\alpha,\beta\in\form\Coloneqq p \mid \bot \mid 
    (\alpha\to\beta) \mid \nec\alpha \mid \alpha\trir\beta$. Working on this grammar, we introduce the following notations:
\begin{center}  
    \begin{tabular}{l l l}
        $\top\coloneqq\bot\to\bot$ & 
        $\neg\alpha\coloneqq\alpha\to\bot$ &
        $ \alpha\wedge\beta\coloneqq\neg(\alpha\to\neg\beta)$ \\ $\alpha\vee\beta\coloneqq(\neg\alpha\to\beta)$ &
        $\alpha\leftrightarrow\beta\coloneqq(\alpha\to\beta)\wedge(\beta\to\alpha)$ &
        $\pos\alpha\coloneqq\neg\nec\neg\alpha$ \\
    \end{tabular}
\end{center}
Furthermore, we adopt the convention of omitting brackets on formulas to enhance their readability, according to the following decreasing priority list: $[\neg,\nec,\pos$; $\wedge,\vee;\trir;\to,\leftrightarrow ]$.\footnote{According to this notation convention, we can thus unambiguously write, e.g., $\alpha\trir\beta \to\alpha \wedge \nec\gamma \trir \beta \wedge \nec\gamma$ for $(\alpha\trir\beta) \to ((\alpha\wedge (\nec\gamma)) \trir (\beta\wedge (\nec\gamma)))$.}

\subsection{Axiomatic Calculus and Relational Semantics for \il}

\begin{definition}[Axiomatization of \il]\label{def:calc}
    The axiomatic calculus \il for the basic interpretability logic extends classical propositional logic by the following axiom schemas and rules:
    \begin{multicols}{2}
    \begin{itemize}
        \item \K: \kripke;
        \item  \GL: \lob;
        \item \play{1}: \interone;
        \item \play{2}: \intertwo;
        \item \play{3}: \interthree;        
        \item \play{4}: \interfour;        
        \item \play{5}: \interfive;
        \item The modus ponens rule: \AxiomC{$\alpha\to\beta$}\AxiomC{$\alpha$}\BinaryInfC{$\beta$}\DisplayProof;
        \item The necessitation rule: \AxiomC{$\alpha$}\UnaryInfC{$\nec\alpha$}\DisplayProof.
    \end{itemize}
    \end{multicols}
    We write $\il\vdash\alpha$ when the formula $\alpha$ is provable in the axiomatic calculus for \il; similarly $\Gamma\vdash_\il\alpha$ denotes provability in \il of $\alpha$ from the \textit{finite} set of hypotheses $\Gamma$, with the usual restriction on the application of the necessitation rule.\footnote{Notice that $\il\vdash\nec\alpha\to\nec\nec\alpha$ and $\il\vdash\nec\alpha\leftrightarrow\neg\alpha\trir\bot$. An equivalent axiomatization of \il can then be given over a restricted language $\mathcal{L}_{\trir}$~\cite{Visser:1990:InterpretabilityLogic}.}
\end{definition}

The system \il has been extended along several directions, to cover different versions of formal interpretability between arithmetical theories: for an in-depth treatment, we refer the reader to, e.g.,~\cite{Shavrukov:1988:InterpretabilityLogicPA,Berarducci:1990:InterpretabilityLogicPA,Visser:1990:InterpretabilityLogic,GorisJoosten:2020:TwoSeries}. On the relational semantic side, \il has a natural presentation in terms of Veltman frames~\cite{JonghVeltman:1990:ProvabilityLogicsForRelativeInterpretability}.

\begin{definition}[Veltman frames and models]\label{def:veltman}
    A \emph{Veltman frame} is a triple $\veltframe:=\la W, R, \{S_w \mid w\in W\}\ra$ where $W$ is a non-empty 
set of \emph{possible worlds}, $R$ a binary relation on $W$ which is transitive and conversely well-founded. The $\{S_w \mid w\in W\}$ is a collection of binary relations on $R[w]$ -- where $R[w]:= \{ v\mid wRv \}$ -- so that each $S_w$ is reflexive and transitive and the restriction of $R$ to $R[w]$ is contained in $S_w$, that is: $R\cap R[w] \subseteq S_w$.

A \emph{Veltman model} consists of a Veltman frame together with a valuation function $\eval: \atom\longrightarrow\powerset{W}$ that assigns to each propositional atom $p\in\atom$ a set of worlds $\eval(p)\subseteq W$ where $p$ is stipulated to be true.

This valuation function naturally extends to a forcing relation $\Vdash\,\subseteq\,W\times\form$ as follows:

\(
\begin{array}{lll}
w\Vdash p & \iff & w\in \eval(p); \ \mbox{ and }   w\Vdash \bot \text{for no}\,w\in W;\\
w\Vdash \alpha\to \beta & \iff & w\nVdash  \alpha \text{ or } w\Vdash \beta;\\
w\Vdash \nec \alpha & \iff & \forall v\ .(\text{if } wRv, \text{ then } v\Vdash  \alpha);\\
w\Vdash \alpha\trir \beta & \iff & \forall u\  .\Big(\text{if } wRu \text{ and } \ u\Vdash \alpha, \text{ then } 
\exists v \ (uS_w v \ \text{ and } \ v\Vdash  \beta)\Big).
\end{array}
\)

Given a Veltman model $\mathcal{M} = \la W, R, \{S_w \mid w\in W\},\eval\ra$ and a formula $\alpha\in\form$, we shall write $\lb \alpha\rb_{\mathcal{M}}$  to denote $\{x\in \mathcal M \mid \mathcal M , x \Vdash \alpha\}$ and,
$\mathcal{M} \models\alpha$ as a notation for $\forall \, w \in W.\, ( w \Vdash\alpha)$. 
\end{definition}

In their~\cite{JonghVeltman:1990:ProvabilityLogicsForRelativeInterpretability}, De Jongh and Veltman proved that the calculus \il is sound and complete with respect to all Veltman models. Modal adequacy w.r.t.~Veltman frames satisfying different conditions  has been established via various techniques in subsequent work on several axiomatic extensions~\cite{JonghVeltman:1990:ProvabilityLogicsForRelativeInterpretability,JonghVeltman:1999:ILW,Joosten:1998:MasterThesis,GorisJoosten:2011:ModalMatters,BilkvaGorisJoosten:2004:SmartLabels,GorisBilkovaJoostenMikec:2020:ArXivLabels}. 

\begin{definition}[Frame definability]\label{def:frameDef}
Let $\mathcal{C}$ denote a first- or higher-order formula. Let us write $\mathcal{F}\models\mathcal{C}$ for stating that the Veltman frame $\mathcal{F}$ as first- or higher-order structure satisfies the condition encoded by $\mathcal{C}$, i.e. it makes $\mathcal{C}$ true. Equivalently, we will write $\mathcal{F}\in\mathcal{C}$ for the class of frames $\Frame F$ such that $\Frame F\models\mathcal{C}$. The class of frames $\mathcal{C}$ is \textbf{definable} if
    $\exists\, \varphi\in \form.\,\forall \mathcal{F}.\,(\mathcal{F}{\in}\mathcal{C} {\iff} \mathcal{F}\vDash\varphi).$
\end{definition}

\subsection{Labelling Techniques for Interpretability Logics}\label{sec:labelling}

In the canonical model for interpretability logic, maximal consistent sets may need copies if they fulfill various functionalities. To keep track of the copies and their functionality in the canonical model, one typically uses labels. In the early times the labels were formulas to flag so-called critical successors (see~\cite{JonghVeltman:1990:ProvabilityLogicsForRelativeInterpretability}). Later the labeling techniques were generalized giving rise to so-called \emph{assuring labels} (\cite{GorisBilkovaJoostenMikec:2022:JournalLabels}).



\begin{definition}[$S$-assuring successor]\label{def:assuringLabel} Given sets of formulas $\Gamma,\Delta,S$, we say that $\Delta$ is an $S$-assuring successor of $\Gamma$ (denoted $\Gamma\prec_S\Delta$) if, for any formula $\alpha\in\form$ and any finite subset of the label $S'\subseteq S$, the condition that $\neg\alpha$ interprets the disjunction of the negations of formulas in $S'$ implies that $\alpha$ and $\nec\alpha$ are in $\Delta$:
    $$\Gamma\prec_S\Delta \iff \forall\alpha\in\form,\,S'\underset{\textit{fin}}{\subseteq}S.\,\Big((\neg\alpha\trir\bigvee_{\beta\in S'}\neg\beta)\in\Gamma\implies\alpha,\nec\alpha\in\Delta\Big).$$
If that is the case, $S$ is a \textbf{label} (for $\Gamma$ and $\Delta$).
\end{definition}
For our aims, it suffices to recall some basic labelling lemmas for generic assuring labels.
For the next lemmas, we assume that $\Gamma$ and $\Delta$ are maximal consistent (with respect to some logic  extending \il) sets of formulas. The following lemma is easy to check.
\begin{lemma}[From~\cite{GorisBilkovaJoostenMikec:2020:ArXivLabels}, Lemma 4.2] The assuring relation $\prec_S$ satisfies the following properties:
    \begin{enumerate}
        \item Restriction: If $S\subseteq T$ and $\Gamma\prec_T\Delta$, then $\Gamma\prec_S\Delta$.
        \item Propagation: If $\Gamma\prec_S\Delta$ and $\Delta\prec_\varnothing\Theta$, then $\Gamma\prec_S\Theta$.
    \end{enumerate}
\end{lemma}
The following lemma tells us that labels can be extended to theories. 
\begin{lemma}[From~\cite{GorisBilkovaJoostenMikec:2020:ArXivLabels}, Lemma 4.6] The assuring relation $\prec_S$ satisfies the following closure properties:
    \begin{enumerate}
        \item Deducibility: For any $\varphi\in\form$, if $\Gamma\prec_S\Delta$ and $S\vdash_\il\varphi$, then $\Gamma\prec_{S\cup\{\varphi\}}\Delta$.
        \item Necessitation: If $\Gamma\prec_S\Delta$, then $\Gamma\prec_{S\cup\nec S}\Delta$, where $\nec S:=\{\nec\varphi \mid \varphi\in S\}$.
    \end{enumerate}
\end{lemma}

\begin{lemma}[From~\cite{GorisBilkovaJoostenMikec:2020:ArXivLabels}, Lemma 3.7]\label{lemma:succes} For any $S$, the following hold:
    \begin{itemize}
        \item If $\Gamma\prec_S\Delta$, then $S,\nec S\subseteq\Delta$.
        \item If $\Gamma\prec_S\Delta$, then $\pos S\subseteq\Gamma$, where $\pos S:=\{\pos\varphi \mid \varphi\in S\}$.
    \end{itemize}
    
\end{lemma}

\subsection{Ultrafilter Extensions}

(Ultra)Filters are the algebraic counterpart of (complete) theories. 


\begin{definition}\label{def:ultrafilter} 
    Given a set $W$, an {ultrafilter} over $W$ consists of a maximal filter over $W$. That is, a collection $f\subseteq\wp(W)$ is an \textbf{ultrafilter} on $W$, and we write $f\in\textsf{U}({W})$, if it satisfies the following conditions:
     \textit{Properness:} $\varnothing\not\in f$;
     \textit{Upward Closure:} If $A\in f$ and $A\subseteq B$, then $B\in f$;
     \textit{Intersection Closure:} If $A,B\in f$, then $A \cap B\in f$;
         \textit{Totality:} For any $A\subseteq W$, either $A\in f$ or $\overline{A}\in f$, where $\overline{A}\coloneqq W\setminus A$.

    Ultrafilters over a given $W$ generated by a single fixed element $w\in W$ by considering all subsets of $W$ containing $w$ are called \textbf{principal ultrafilters}, denoted by $\Pi_w$. 
\end{definition}

When the starting set $W$ is infinite, the existence of non-principal ultrafilters over $W$ is provable modulo some choice principles over a constructive base~\cite{leinster2013codensity,BellSlomson:2006:ModelsAndUltraproducts} as a consequence of

\begin{lemma}[Ultrafilter Principle]\label{lemma:ultrafilterPrinciple} Every non-empty family $s$ of sets with the finite intersection property (FIP) -- that is, for any $t\underset{\textit{fin}}{\subseteq}s$, $\bigcap t\neq\varnothing$ -- generates a filter that is contained in at least one ultrafilter.
\end{lemma}



Frames can be extended via ultrafilters as follows.

\begin{definition}\label{def:ultrafilterExtension}
    Given a relational frame $\mathcal{F} = \langle W, \{R_i\}_{i <\omega} \rangle$, its \textbf{ultrafilter extension} $\ultraFilterCounterpart{\mathcal{F}}$ is given by:
    \begin{itemize}
        \item \textbf{Worlds:} The new domain is $\textsf{U}({W})$, the set of all ultrafilters over $W$;
        \item \textbf{Accessibility Relations ($\ultraFilterCounterpart{R}_i$):} For any two ultrafilters $f, g \in \textsf{U}({W})$:
        $$f \ultraFilterCounterpart{R}_i g \iff \forall X \subseteq W.\,(X \in g \implies R^{-1}_i(X)\in f),$$    where $R^{-1}_i(X) := \{ w \in W \mid \exists x \in X.\, (w R_i x) \}$.
    \end{itemize}
    Given such an ultrafilter extension $\ultraFilterCounterpart{\Frame F}$, and an \textbf{evaluation} $\eval$ over $\Frame F$ we define
    $\ultraFilterCounterpart{\eval}(p):=\{f\in\textsf{U}({W}) \mid \eval(p)\in f\}.$
    Therefore, given a model $\mathcal{M}$, we denote by $\ultraFilterCounterpart{\mathcal{M}}$ its ultrafilter extension. 
\end{definition}

Working with the standard forcing relation $x\Vdash \langle i\rangle \varphi :\Leftrightarrow \exists \, y\ (xR_iy \wedge y \Vdash \varphi)$, the central results on ultrafilter extensions of relational models are collected by (see, e.g.~\cite{BlackburnEtAll:2001:ModalLogic})

\begin{theorem} Given a relational model $\Frame M=\la \Frame F, \eval\ra$, we have:
    \begin{itemize}
        \item \textbf{Truth Lemma:} For any formula $\varphi$ and any ultrafilter $f\in\textsf{U}({W})$, the following equivalence holds: 
        $(\ultraFilterCounterpart{\mathcal{M}}, f) \Vdash \varphi \iff \eval(\varphi) \in f$,
        where $\eval(\varphi) = \{ w \in W \mid (\mathcal{M}, w) \Vdash \varphi \}$.
        \item \textbf{Reflection of Validity:} If a formula $\varphi$ is valid in the ultrafilter extension, it is valid in the original frame, that is: if $\ultraFilterCounterpart{\mathcal{F}} \vDash \varphi$, then $\mathcal{F} \vDash \varphi$.
        \item \textbf{The Goldblatt-Thomason Theorem:} An elementary (i.e.~first-order definable) class of frames $\mathcal{C}$ is modally definable if and only if it is closed under: generated subframes, bounded morphisms, and disjoint unions; and \textbf{reflects ultrafilter extensions}, that is:
        $\text{if } \ultraFilterCounterpart{\mathcal{F}} \in \mathcal{C}, \text{ then } \mathcal{F} \in \mathcal{C}$.
    \end{itemize}    
\end{theorem}

\section{A Non-Definability Result}\label{section:NonDefinability}
In \cite{NavarroJoosten:2025:tamesemantics} the class of \textit{Pencil frames} $\mathcal{C}_{\textsf{Pencil}}$ is identified by the first-order frame property 
\[
xRyS_xzRu \,\wedge\, yRvS_xu\to yRu.
\]

The modal logics \ilm and \ilp arise by adding the axiom schemes $\alpha \trir \beta \to \alpha \wedge \nec \gamma \trir \beta \wedge \nec \gamma$ and $\alpha \trir \beta \to \nec (\alpha \trir \beta)$ respectively. It is well-known \cite{Visser:1997:OverviewIL} and easy to see that the respective frame conditions are $yS_xzRu \to yRu$ and $xRyRzS_xu \to zS_yu$. The pencil frame condition $\Frame C_{\textsf{Pencil}}$ can be seen to hold on all \ilp and all \ilm frames. Since \ilm is the interpretability logic of all arithmetical theories with full induction \cite{Berarducci:1990:InterpretabilityLogicPA, Shavrukov:1988:InterpretabilityLogicPA} and, since \ilp is the interpretability logic of all finitely axiomatizable theories that prove the totality of super-exponentiation\cite{Visser:1990:InterpretabilityLogic}, the pencil frame is of interest to the long-standing open problem to determine the interpretability logic of all reasonable arithmetical theories \cite{Joosten:2015:TwoSeries, JoostenVisser:2000:IntLogicAll}.

\begin{definition}\label{def:bisimulation}
    [From~\cite{Visser:1997:OverviewIL}]: A relation $Z$ between the nodes of two models $\Frame M$ and $\Frame N$ (denoted $m,m_1,m_2...$ and $n,n_1,n_2...$ respectively) is a bisimulation if and only if it satisfies the following:
    \begin{enumerate}
        \item \label{def:bisimulation1} for all $p\in\atom$ $mZn\Rightarrow (\Frame M,m\Vdash p\iff \Frame N,n\Vdash p)$.
        \item \label{def:bisimulation2} (Zig) If $mR^{\Frame M}m_1$ and $mZn$, then there is some $n_1$ such that $nR^{\Frame N}n_1$ and $m_1Zn_1$ and for all $n_2$ in $\Frame N$ such that $n_1S^{\Frame N}_nn_2$ there is some $m_2$ such that $m_2Zn_2$ and $m_1S^{\Frame M}_mm_2$.
        \item \label{def:bisimulation3} (Zag) If $nR^{\Frame N}n_1$ and $mZn$ then there is some $m_1$ such that $mRm_1$ and $m_1Zn_1$ and for all $m_2$ such that $m_1S^{\Frame M}_mm_2$ there is some $n_2$ such that $m_2Zn_2$ and $n_1S_nn_2$.
    \end{enumerate}
\end{definition}

As always, bisimilar worlds are modally equivalent \cite{Visser:1997:OverviewIL} and this is the only property we use in our proof below. We prove that the class $\mathcal{C}_{\textsf{Pencil}}$ is not modally definable. To do so, let us consider the frames $\mathcal{F}_0$ and $\mathcal{F}_1$ from Figures~\ref{fig:f0} and ~\ref{fig:f1}, respectively, where $\mathcal{F}_0$ and $\mathcal{F}_1$ satisfy the condition $x_iRy_iS_{x_i}z_iRu_i\wedge y_iRv_i\wedge \bigwedge_{j\in\omega} v_iS_{x_i}w^j_i$, but $\mathcal{F}_1$ satisfies $v_1S_{x_1} u_1$ and $\mathcal{F}_0$ does not satisfy $v_0 S_{x_0} u_0$.
\begin{figure}[htbp]
    \centering
    \begin{subfigure}[b]{0.35\textwidth}
        \includegraphics[width=\linewidth]{F0_Reworkedpdf.pdf}
        \caption{The $\mathcal{F}_0$ frame satisfying $\mathcal{C}_{\mathsf{Pencil}}$.}
        \label{fig:f0}
    \end{subfigure}
    \hfill  
    \begin{subfigure}[b]{0.35\textwidth}
        \includegraphics[width=\linewidth]{F1_Reworkedpdf.pdf}
        \caption{The $\mathcal{F}_1$ frame not satisfying $\mathcal{C}_{\mathsf{Pencil}}$.}
        \label{fig:f1}
    \end{subfigure}
    \caption{$\mathcal{F}_0$ and $\mathcal{F}_1$ frames. Only $S_{x_0} $ and $S_{x_1}$ are displayed for clarity. $R-$transitivity and $S-$reflexivity, etc.~are also not displayed.}
\end{figure} 

More precisely, we prove the following.
    \begin{theorem} The frame class identified by the Pencil Condition $\Frame C_{\textsf{Pencil}}$ is not modally definable, that is:
        $\neg\exists \varphi\in\form. \,\forall\mathcal{F}\,(\mathcal{F}\vDash\varphi\iff \mathcal{F}\in\mathcal{C}_{\mathsf{Pencil} }).$
    \end{theorem}
    \begin{proof}
        Suppose for a contradiction that $\exists \varphi\in\form.\,\forall\mathcal{F}(\mathcal{F}\vDash\varphi\iff \mathcal{F}\in \mathcal{C}_{\mathsf{Pencil} }).$
        Then, pick the frames $\mathcal{F}_0$ and $\mathcal{F}_1$ depicted in Figures \ref{fig:f0} and \ref{fig:f1}, respectively. So, the frames are identical but for one $S$ relation: $v_1S_{x_1}u_1$ holds whereas $v_0S_{x_0}u_0$ does not. We observe that $v_iS_{x_i}w^j_i$, and there is no $S_{x_i}$ relation between any $w^j_i$ and $w^k_i$ for different $j\neq k>0$. Also, we observe that $\mathcal{F}_0\in \mathcal{C}_{\mathsf{Pencil}}$ and, $\mathcal{F}_1\notin\mathcal{C}_{\mathsf{Pencil}}$. Therefore, by hypothesis, $\mathcal{F}_0\vDash \varphi$ and $\mathcal{F}_1\nvDash \varphi$.
        
        Now, we prove the following claim:
        For every valuation $\eval_1$ on $\mathcal{F}_1$ there is another valuation $\eval_0$ on $\mathcal{F}_0$ such that $\Frame M\coloneqq \langle \mathcal{F}_1, \eval_1\rangle$ and $\Frame N\coloneqq \langle \mathcal{F}_0, \eval_0\rangle$ are bisimilar.

            Indeed, consider an arbitrary $\mathcal{F}_1$-valuation $\eval_1$ and consider the model $\langle\mathcal{F}_1,\eval_1\rangle$. 
            Given $p\in\atom$, let us define $\eval_0(p)$ over $\Frame F_0$ as follows:
            \begin{multicols}{2}
            \begin{itemize}
                \item For $t\in\{x,y,z,u\}$:
                
                $t_0\in \eval_0(p):\iff t_1\in \eval_1(p)$
                \item$v_0\in \eval_0(p):\iff v_1\in \eval_1(p),$
                \item$w^1_0\in \eval_0(p):\iff u_1\in \eval_1(p),$
                \item$w^{i+1}_0\in \eval_0(p):\iff w^i_1\in \eval_1(p) \quad (i\geq 1).$
            \end{itemize}
            \end{multicols}
            Next, consider the relation $Z\coloneqq\{(x_1,x_0),(y_1,y_0),(z_1,z_0),(v_1,v_0),(u_1,u_0),(u_1,w_0^1)\}\cup\{(w_1^{i},w_0^{i+1})\}_{i\geq 1}.$
            It is routine to check that $Z$ is a bisimilation between $\langle\mathcal{F}_1,\eval_1\rangle$ and $\langle\mathcal{F}_0,\eval_0\rangle$.
            First we check that for all $p\in\atom$,  $mZn\Rightarrow (\Frame M,m\Vdash p\iff \Frame N,n\Vdash p)$. This is clear by the definition of $ev_0$ provided above.
            Next, the Zig condition \ref{def:bisimulation}.\ref{def:bisimulation2}. We note that $x_1R^1[W_1-\{x_1\}]$, that $x_0R^0[W_0-\{x_0\}]$, that $x_1Zx_0$ and, of course, that for all $t_1\in W_1$ there is some $t_0\in W_0$ such that $t_1Zt_0$. For the other points we see that $z_1R^1u_1$ for $z_1Zz_0$ and $z_0R^0u_0$ for $u_1Zu_0$. Also, observe that $y_1R^1w_1^{i}$ for all $i$, and $y_1 Z y_0$, but notice that $y_0R^0w_0^{i}$ for all $i$ and that for all $w_1^i$ there is some $w_0^j$ such that $w_1^iZw_0^j$. Besides of that, $y_1R^1v_1$ and $y_1Zy_0$, but observe that $y_0R^0v_0$ and $v_1Zv_0$.

            $S_w$ for $w\in W_1-\{x_1\}$ our claim also holds following the definition. The only critical point is for the $S^1_{x_1}$ relation, for which we will only show that condition \ref{def:bisimulation}.\ref{def:bisimulation2} holds for $v_1S^1_{x_1}u_1$ and $v_1S^1_{x_1}w_1^i$, for $i\geq 0$. We observe that $x_1 R^1 v_1$, $x_1 Z x_0$, $x_0 R^0 v_0$ and $v_1 Z v_0$. Also, notice that $v_0 S^0_{x_0} w_0^j$, for $j \geq 0$. We need to find some $m_2 \in W_1$ such that $m_2 Z w_0^j$ and $v_1 S_{x_1} m_2$. It suffices to pick $m_2= v_1$, $m_2 = u_1$ and $m_2 = w_1^{j-1}$ for the cases $j=0$, $j=1$ and $j>1$, respectively.
            
            Condition $\ref{def:bisimulation}.\ref{def:bisimulation3}$ follows in a similar way. 
            
        %
        Since $\mathcal{F}_1\nvDash \varphi$, we have that there exists a valuation $\eval$ such that $\langle\mathcal{F}_1,\eval\rangle\nvDash\varphi$. By the last claim, there is a valuation $ev'$ such that $\la\mathcal{F}_1,\eval\ra$ is bisimilar with $\la\mathcal{F}_0,\eval'\ra$. By preservation of validity under bisimulation~\cite[§~3.4]{BLACKBURN20071}, we have that $\langle\mathcal{F}_0,\eval'\rangle\nvDash \varphi,$
        contra $\mathcal{F}_0\vDash \varphi$.
    \end{proof}

\section{Ultrafilter Extensions}\label{section:ultrafilterExtensions}

To define ultrafilter extensions for interpretability logic we need to find algebraic counterparts for labels and assuring successors.

\subsection{Theories, Filters and Ultrafilters}
In this section we recall some basic facts of modal ultrafilter extensions and formulate the basic definitions needed to generalise all to interpretability logic.
\begin{definition}
Given a binary relation $R$ over non-empty $W$, we define 
$\Rhat : \powerset W \longrightarrow \powerset W$ as the dual operator for $R^{-1}$ (from Def.~\ref{def:ultrafilterExtension}) as follows: $
\Rhat (Y) \ := \ \{ x \in W \mid \forall y (xRy \Rightarrow y \in Y) \}$. Moreover, we define $f_{\nec} := \{ Y \mid \Rhat(Y) \in f\}$.

\end{definition}

The following lemmas are folklore -- see, e.g., \cite{BlackburnEtAll:2001:ModalLogic}.
\begin{lemma}\label{lemma:interdef}
    Given a a binary relation $R$ over a non-empty $W$, for any $Y\subseteq W$ we have
    $\Rhat(Y){=}\overline{R^{-1}(\overline{Y})}$ and $R^{-1}(g)\subseteq f \ \Longleftrightarrow \ f_{\nec} \subseteq g$.
\end{lemma}

In our context, we need to define filter-assuring successor, in analogy with Def.~\ref{def:assuringLabel} and \cite{GorisBilkovaJoostenMikec:2020:ArXivLabels}.

\begin{definition}
    
 Given a Veltman frame $\Frame F = \langle W,R, \{ S_x \mid x\in W \}\rangle$, let us define\footnote{In \cite{sestak2024generalframesinterpretability} (corresponding to \cite{Sestak2024Generalframesforinterpretabilitylogic}) the notation $m_{\trir}$ is used instead of $S^{-1}$.}
$S^{-1} : \powerset{W}\times \powerset{W} \longrightarrow \powerset{W}$ by
    $$S^{-1}(X,Y):= \{w{\in} W\mid \forall x{\in} X.\,(wRx\Rightarrow \exists y{\in}Y.\,xS_wy) \}.$$
    
Let $l\subseteq\wp(W)\setminus \{ \varnothing\}$  have the FIP, and let $f,g\in\textsf{U}(W)$. We define $f\prec_l g$ as follows:
    \[
    f\prec_l g \ \ \: \iff \Bigg[ \forall A{\subseteq} W\, \forall^{\sf fin} \{ S_i \in l\}\, \Big( S^{-1}(\overline A, \bigcup_{i}\overline{S_i})\in f \Rightarrow A, \Rhat A \in g \Big) \Bigg]
    \]
\end{definition}

Lemma \ref{lemma:assuringThenSuccessors} below will prove that filter-assuring successors are indeed successors in the sense of $R^{\sf ue}$.

\subsection{Ultrafilter Extensions for Frames and Models}\label{sec:ultrafilterVeltman}

\begin{definition}\label{def:uedef} Let $\Frame F=\langle W,R,\{S_w \mid w\in W\}\rangle$ be a Veltman Frame, the ultrafilter extension of the Veltman frame $\ultraFilterCounterpart{F}$ is defined recursively:
\begin{enumerate}
    \item \label{sec:ultrafilterVeltman1}
    $\ultraFilterCounterpart{W}$ is recursively defined:
    \begin{enumerate}
        \item $(f,\langle\rangle)\in \ultraFilterCounterpart{W}, \text{ for each ultrafilter $f$ over $W$}$.
        \item $ (f,\sigma)\in \ultraFilterCounterpart{W} \wedge f\prec_l g\Rightarrow( g,\sigma^\smallfrown \la l \ra ) \in \ultraFilterCounterpart{W}$.
    \end{enumerate}
    \item \label{sec:ultrafilterVeltman2}
    $\ultraFilterCounterpart{R}$ is defined as the transitive closure of $\ultraFilterCounterpart{R}_{\sf one}$ where $\ultraFilterCounterpart{R}_{\sf one}$ is defined for $\la f,\sigma\ra,  \la g, \sigma\smallfrown \la l \ra\ra \in \ultraFilterCounterpart{W}$ in the natural way as:
    \( \la
    f,\sigma \ra \ultraFilterCounterpart{R}_{\sf one}\la g,\sigma^\smallfrown \la l\ra \ra \ \iff \ f \prec_l g
    \)
    
    \item \label{sec:ultrafilterVeltman3}
    $S^{\sf ue}_{\la f,\sigma\ra}$ is defined as the smallest relation that is reflexive, transitive, so that it contains\\
    $R^{\sf ue}[\la f,\sigma \rangle]^2\cap R^{\sf ue}$ and $\ultraFilterCounterpart{S}_{\la f,\sigma \ra, {\sf one}}$  where the latter\footnote{We take the closure under the one-step relation here. Another approach can be to consider slightly different semantics as in \cite{NavarroJoosten:2025:tamesemantics}.} is defined as follows 
    \[
    \la g,\tau \ra \ultraFilterCounterpart{S}_{\la f,\sigma \ra, {\sf one}}\la h,\tau'\ra \iff 
\begin{cases}
\la f,\sigma \ra \ultraFilterCounterpart{R}\la g,\tau \ra \\
\la f,\sigma\ra \ultraFilterCounterpart{R}\la h,\tau'\ra \\
(\tau)_{|\sigma|}=(\tau')_{|\sigma|} \\
\end{cases}
\]
Note in particular $\tau = \sigma\smallfrown\la l\ra\smallfrown \overline \tau$ for some $l$ and $\overline \tau$ and that $\tau' = \sigma\smallfrown\la l\ra\smallfrown \overline{\tau'}$ for the very same $l$ and some $\overline{\tau'}$.
Furthermore, we have $f\prec_l g, h$ (See Lemma \ref{RlTransitivity}).
\end{enumerate}
\end{definition}
\begin{remark}
    A few considerations about this type of frames:\\
    \begin{tabular}{l}
        \makecell[l]{1. $\;$ Labels are a very simple way to keep track of the history of each world, which makes \\it possible to ``localize'' (in the sense of making local) the $\ultraFilterCounterpart{S}$ relation and the $\trir$ operator.}\\
        2. $\;$ Finite frames are no longer necessarily isomorphic to their ultrafilter extensions.
            \end{tabular}
\end{remark}

Next we define a translation between \form~and the language of ultrafilter extensions:

\begin{definition} For any formula $\varphi\in \form$ we define a translation $\varphi^t$  inductively by \\
    $\begin{array}{l}
        \quad\quad\quad\cdot\bot^t\coloneqq \varnothing \qquad\quad\quad
        \cdot p^t := A_p \text{where } p \in \atom \text{ and } A_p \text{ a set-variable}\quad
        \cdot(\varphi \to \psi)^t := \overline{(\varphi^t)}\cup (\psi^t) \\
        \qquad\qquad\qquad\qquad\cdot (\nec \varphi)^t := \Rhat (\varphi^t)\qquad\qquad
        \cdot(\varphi \trir \psi)^t :=  S^{-1}\Big( (\varphi^t), (\psi^t)\Big)
    \end{array}$
\end{definition}
We observe that $(\nec \phi)^{t} = \{x \in W: R[x] \subseteq \varphi^{t}\}$ and $(\pos \varphi)^{t} = \{x \in W: R[x] \cap \varphi^{t} \neq \emptyset\}$.

For a Veltman frame $\Frame F=\la W, R,\{S_x\}_{x\in W} \ra$, once we fix a valuation $\sf e$ that assigns subsets of $W$ to the set-variables $A_p$, this fixes in the obvious way for each formula $\varphi$ a set that we denote by $\llbracket \varphi ^t\rbf$. Similarly we use $p_A$ to denote the propositional variable such that $\lb p\rbf=A$.

Given a set of formulas $\Gamma=\{\varphi_1,...,\varphi_n\}$ we use $\lb\Gamma\rbf$ to denote the set $\{\lb \varphi\rbf,...,\lb \varphi\rbf\}$.



\begin{theorem}\label{thm:provW}
    Let $\Frame F = \la W, R,\{S_x\}_{x\in W} \ra$ be a Veltman frame. If $\il \vdash \varphi$, then for any valuation $\sf e$ we have $\lb \varphi^t \rbf {=} W$.
\end{theorem}

\begin{proof}
    Throughout the proof we fix some frame $\Frame F = \la W, R,\{S_x\}_{x\in W} \ra$ and all sets mentioned below are subsets of $W$. Let us fix a choice of a valuation $\sf e$. We first need to prove that the translation of all the $\mathbf{IL}$ axioms under $\sf e$ evaluate to $W$, that is, $\llbracket \varphi^t \rbf = W$ for every axiom $\varphi$ of \il.  We only prove this for axioms $\GL$ and $\play{2}$ as the remaining axioms follow the same lines of algebraic reasoning:
    \begin{enumerate}

        \item 
        ${\sf GL}:\nec(\nec \alpha\rightarrow \alpha)\rightarrow \nec \alpha$; the axiom translates to $\varphi^t=\overline{\Rhat(\overline{\Rhat(\alpha^t)}\cup \alpha^t)}\cup \Rhat(\alpha^t)$. To see that, for any choice of $\alpha$, $\llbracket \varphi^t \rbf = W$, it suffices to show that, for any $A\subseteq W$, we have $\Rhat(\overline{\Rhat(A)}\cup A)\subseteq \Rhat(A)$.  
        Assume $x\in \Rhat(\overline{\Rhat(A)}\cup A)$ and assume $xRy$ and $y\notin A$. Since $R$ is transitive and conversely well-founded, let $z$ be the $R$-maximal element in $W$ such that $z\in \{w\in W \mid (yR w\wedge w\notin A) \vee w=y\}$. Since $xRz$ and neither $z\in A$ nor $z\in\overline{\Rhat(A)}$, we have a contradiction. Therefore, $y\in A$ for our arbitrary $y$, and we conclude $x\in\Rhat(A)$.       

        \item 
        $\play{2}:(\alpha\trir\beta\wedge\beta\trir\gamma)\rightarrow \alpha\trir\gamma$; this axiom translates to $\varphi^t=\overline{S^{-1}(A,B)\cap S^{-1}(B,C)} \cup S^{-1}(A,C)$; to see that $\llbracket \varphi^t \rbf = W$, it suffices to show that $S^{-1}(A,B)\cap S^{-1}(B,C)\subseteq S^{-1}(A,C)$.
        Assume thus $x\in S^{-1}(A,B)\cap S^{-1}(B,C)$ and there is some $y\in W$ such that $xRy$ and $y\in A$. Then, there is some $z\in W$ such that $yS_xz\wedge z\in B$. Since $yS_xz\rightarrow xRy\wedge xRz$, there is some $v\in W$ such that $xRv$ and, by the definition of $S^{-1}$, we have that $zS_xv$ and $v\in C$. Since $yS_xz\wedge zS_xv\rightarrow yS_x v$, we conclude $x\in S^{-1}(A,C)$.
        
        
        
    \end{enumerate}
        We now prove that the rules of \il preserve the property of  derivable formulas evaluating to $W$ under $\sf e$:
    \begin{enumerate}
        \item Modus Ponens: By inductive hypothesis on $\mathbf{IL}\vdash \varphi$ and $\mathbf{IL}\vdash\varphi\rightarrow \psi$, we have $\lb\varphi^t\rbf=W$ and $\lb\overline{\varphi^t}\rbf\cup \lb\psi^t\rbf=W$, so that $\psi^t=W$.
        \item Necessitation: By inductive hypothesis on $\mathbf{IL}\vdash \varphi$, we have $\lb\varphi^t\rbf=W$, so that, for any $x,y\in W$, $xRy\Rightarrow y\in \lb\varphi^t\rbf=W$. Hence we conclude $\lb\nec\varphi^t\rbf=\Rhat(\varphi^t)=W$.
    \end{enumerate}    
\end{proof}

\begin{corollary}\label{cor:deductionThm}
    Let $\mathcal{F}=\langle W,R\{S_x\}_{x\in W}\rangle$ a Veltman frame. If $\il\vdash \varphi\to\psi$
    then, for every valuation ${\sf e}$, $\llbracket\varphi^t\rbf\subseteq\llbracket\psi^t\rbf.$
\end{corollary}
\begin{proof}
    By Theorem \ref{thm:provW}, we have $\llbracket(\varphi\to\psi)^t\rbf=\overline{\llbracket\varphi^t\rbf}\cup\llbracket\psi^t\rbf=W$. Then, for any $x\in\llbracket \varphi^t\rbf\subseteq W$, since $x\in W$, 
    $x\in\llbracket\psi^t\rbf$. 
    Therefore, 
    $\llbracket\varphi^t\rbf\subseteq\llbracket\psi^t\rbf$.
\end{proof}

\begin{lemma} \label{infforallevaluationsW}
    Let $\mathcal{F}=\langle W,R,\{S_x\}_{ x\in W}\rangle$ be a Veltman Frame, $\varphi\in\form$, and $\sf e$ a valuation on \Frame F.  If, for every ultrafilter $f\in\textsf{U}(W)$, $\lb \varphi^t\rbf \in f$, then $\lb\varphi^t\rbf = \lb\nec\varphi^t\rbf =W$.
\end{lemma}
\begin{proof}
        Since ${\bigcap}{\textsf{U}(W)}=\{W\}$, we have $\lb\varphi^t\rbf \in \{W\}$, that is, $        \lb\varphi^t\rbf=W$. But $\lb\nec\varphi^t\rbf=\Rhat(\lb\varphi^t\rbf)=\Rhat(W)=W$, and we are done.
\end{proof}


Since any \il-tautology evaluates to $W$ for any {\sf e} on any frame \Frame F, we can use \il reasoning in the ultrafilter setting, as stated by the following
\begin{lemma}\label{lemma:implication}
    Given a Veltman frame $\mathcal{F}=\la W, R, \{S_x\}_{x\in W} \ra$ and a valuation ${\sf e}$ over $W$, if $\Gamma \vdash \varphi$ and $\lb\Gamma^t\rbf \subseteq f$, then $\lb\varphi^t\rbf \in f$.
\end{lemma}

\begin{proof}
   Let $\Gamma\vdash\varphi$ and {\sf e} be such that $\lb \Gamma^t\rbf\subseteq f$.
   By the deduction theorem, we have that   $\il\vdash\bigwedge_i(\gamma_i)\rightarrow \varphi$ for $\gamma_i\in\Gamma$; by Theorem~\ref{thm:provW}, we thus have $\lb \bigwedge_i(\gamma_i)\rightarrow \varphi\rbf=W$, that is: $\lb \overline{\bigwedge_i(\gamma_i)^t}\cup \varphi^t\rbf=W$. Therefore, $\lb\bigwedge_i(\gamma_i)^t\rbf=\bigcap(\lb\gamma_i^t\rbf)\subseteq \lb\varphi^t\rbf$.
      Since $f$ is closed under finite intersections and, by assumption,
    $\lb\gamma_i\rbf\in f$ for each $\gamma_i\in\Gamma$, we obtain $\bigcap(\lb \gamma_i^t\rbf)\in f$; and since $f$ is closed under supersets, we conclude that $\lb \varphi^t\rbf\in f$.  
\end{proof}

\begin{corollary}\label{theorem:AlgebraicVersionsOfSimpleILfacts}
    The following hold:
    \begin{enumerate}
        \item \label{item:AtrirBandBtoCThenAtrirC:theorem:AlgebraicVersionsOfSimpleILfacts}
        $S^{-1}(X,Y)\subseteq S^{-1}(X,Y\cup Z)$;

         \item  \label{item:necAnecnecA:theorem:AlgebraicVersionsOfSimpleILfacts}
          $\Rhat(A)\subseteq \Rhat(\Rhat(A))$.

          \item $S^{-1}(A,B)\cap S^{-1}(B,C)\subseteq S^{-1}(A,C).$
    \end{enumerate}
\end{corollary}

\begin{proof}
    All items follow easily from Theorem \ref{thm:provW} either directly or indirectly. For example, Item \ref{item:AtrirBandBtoCThenAtrirC:theorem:AlgebraicVersionsOfSimpleILfacts} reflects that $(\varphi\trir \psi)\wedge (\psi \to \theta) \to (\varphi\trir \theta)$.
\end{proof}

Finally, we extend the ultrafilter extensions of frames into ultrafilter extensions of models in a natural way:
\begin{definition}\label{definition:ultrafilterExtensionModel}
    Given a Veltman model $\Frame M=\la \Frame F, \eval\ra$, we define its ultrafilter extension $\ultraFilterCounterpart{\Frame M}$ to be the pair $\la \ultraFilterCounterpart{\Frame F}, \ultraFilterCounterpart{\eval}\ra$ where $\ultraFilterCounterpart{\Frame F}$ is as in Def.~\ref{def:uedef}, and $\ultraFilterCounterpart{\eval}$ is such that for every $\la f,\ra\sigma\ra\ra$ we have $\la f,\la \sigma\ra\ra\in\ultraFilterCounterpart{\eval}(p) \ \Longleftrightarrow \  \eval(p)\in f$.
\end{definition}

\section{Full Labels in Algebras}\label{section:fullLabelsAlgebras}
We translate here the labelling machinery from Sect.~\ref{sec:labelling} to the setup from Sect.~\ref{sec:ultrafilterVeltman}.

\begin{lemma}\label{lemma:assuringThenSuccessors}
    If $f\prec_l g$ then $R^{-1}g \subseteq f$. 
\end{lemma}

\begin{proof}
    Assume that $f\prec_l g$, for some $f,g\in\textsf{U}(W)$, and $l$ is a proper filter over $W$.
    By Lemma \ref{lemma:interdef}, we
    consider some $A\subseteq W$ such that $\Rhat (A)\in f$. We know that $\nec p\vdash_\il\neg p \trir \bot$, for any $p\in\atom$. Now, define a valuation ${\sf e}$ such that $\lb p_A^t\rbf=A$. Therefore, $\llbracket(\nec p_A)^t\rbf=\Rhat(A)\in f$ and, by Lemma \ref{lemma:implication}, $\llbracket(\neg p_A^t\trir \bot)^t\rbf=S^{-1}(\overline{A},\varnothing)\in f$. Since $f\prec_l g$, we conclude $A, \Rhat(A)\in g$.
\end{proof}

Assuring successors are successors as per Lemma~\ref{lemma:succes}:

\begin{lemma}\label{lemma:assuringAssures}
    If $f\prec_l g$ then $l, \Rhat(l)\subseteq g$.
\end{lemma}
\begin{proof}
    Let $A\in l$. Since $\il\vdash \neg\alpha\trir\neg\alpha$, by Lemma \ref{thm:provW}, we have $S^{-1}\big(\overline{A},\overline{A}\big)=W\in f$. Because $f\prec_l g$, we conclude that $A,\Rhat(A)\in g$.
\end{proof}
 We define, from now on, the convention that whenever a function is considered on a domain, it is naturally extended to the powerset of its domain. In this case, $\Rhat$ on $l$ belongs to the powerset of the domain of $\widehat{R^{-1}}$.
\begin{corollary}
   If $f\prec_l g$, then $R^{-1}(l)\subseteq f$. 
\end{corollary}
\begin{proof}
    Since $f$ is an ultrafilter, for any $A\in l$, either $R^{-1}(A)\in f$ or $\overline{R^{-1}(A)}\in f$; that is, by Lemma \ref{lemma:interdef}, either $R^{-1}(A)\in f$ or $\Rhat(\overline{A})\in f$. However, by Lemma \ref{lemma:assuringThenSuccessors}, from the latter case we would have $\overline{A}\in g$, contra Lemma \ref{lemma:assuringAssures} which tells us that $A\in g$. We thus conclude $R^{-1}(A)\in f$.
\end{proof}
\begin{lemma}\label{LTransitivity}
    Let $\Frame F=\langle W,R,\{S_w \mid w\in W\}\rangle$ be a Veltman frame and $f,g,h\in\textsf{U}(W)$. If $f\prec_l g$ and $g \prec_m h$, then $f\prec_l h$.
\end{lemma}
\begin{proof}
    Assume $f\prec_l g$ and $g \prec_m h$. By Lemma \ref{lemma:assuringAssures}, $\Rhat(l)\subseteq g$. Now, let $A\in l$ and ${\sf e}$ such that $\lb p_A^t\rbf=A$. Since $\il\vdash\nec p_A\rightarrow \nec \nec p_A$, by Lemma \ref{cor:deductionThm}, we have that $\Rhat(\Rhat(A))=\lb \nec\nec p_A\rbf\in g$. Since $g\prec_m h$, we know that $f_{\nec}\subseteq h$, and conclude that $A,\Rhat(A)\in h$.
\end{proof}
 
\begin{lemma} \label{RlTransitivity} Let $\Frame F = \langle W,R, \{ S_x \mid x\in W \}\rangle$ be a Veltman frame. Let $\ultraFilterCounterpart{\Frame F}$ be the ultrafilter extension of $\Frame F$ and $f,g\in \ultraFilterCounterpart{W}$. If $\la f,\sigma\ra \ultraFilterCounterpart{R}\la g,\tau\ra$, then $\sigma $ is an initial segment of $\tau$\footnote{Denoted by $\sigma\subset \tau$ from now on.}  and $f\prec_l g$ for $l$ such that $\tau=\sigma^\smallfrown\la l\ra^\smallfrown \tau'$ proper filter over $W$.
\end{lemma}
\begin{proof}
        By definition,  $\sigma \subset \tau$. By Definition \ref{def:uedef} there is a chain of worlds $\la h_1,\sigma\smallfrown\la l\ra\ra\ultraFilterCounterpart{R}_{\sf one}...\ultraFilterCounterpart{R}_{\sf one} \la g, \tau\ra$. Then we can apply Lemma \ref{LTransitivity} to obtain that $f\prec_l g$ as required.
\end{proof}
The following reflects Lemma 3.7 in \cite{GorisBilkovaJoostenMikec:2020:ArXivLabels} about consistency formulas in labels at the ultrafilter setup:


\begin{lemma}\label{lemma:closureUnderRhat}
    Let $\Frame F=\la W, R, \{S_x\}_{x\in W}\ra$ be a Veltman frame, $f\in\textsf{U}(W)$, and $l$ be a proper filter over $W$. Let $
    B_f\coloneqq \{  A\subseteq W \mid \exists^{\mathrm{fin}}\{S_i\in l\}\Big(S^{-1}\big(\overline{A},\bigcup_i\overline{S_i}\big)\in f \Big)\}
    $.
    Then, if $C\in B_f$,
    $\Rhat(C)\in B_f.$
\end{lemma}
\begin{proof}
    Assume $C\in B_f=\{  A\subseteq W \mid \exists^{\mathrm{fin}}\{S_i\in l\}\Big(S^{-1}\big(\overline{A},\bigcup_i\overline{S_i}\big)\in f \Big)\}$. Then, for some $\{S_i\}\underset{\textit{fin}}{\subseteq}l$, we have that $S^{-1}(\overline{C},\bigcup_i \overline{S_i})\in f$. We need to show that there exists $\{S_j\}\underset{\textit{fin}}{\subseteq}l$ $S^{-1}(\overline{\Rhat(C)},\bigcup_j \overline{S_j})\in f$. Define then ${\sf e}$ such that $\lb p_C^t\rbf=C$ and $\lb p_{S_k}^t\rbf=S_k$ for each $S_k\in l$. By axiom \play{2}, $\il\vdash (\pos\neg p_C\trir \neg p_C \wedge\neg p_C\trir \bigvee_i(\neg p_{S_i}))\rightarrow \pos\neg p_C\trir \bigvee_i(\neg p_{S_i})$. Since $S^{-1}(\overline{C},\bigcup_i(\overline{S_i}))\in f$ and $\lb  \pos\neg p_C^t\trir \neg p_C^t\rbf=W\in f$, we have $\lb \neg p_C^t\trir \bigvee_i(\neg p_{S_i}^t)\wedge \pos\neg p_C^t\trir \neg p_C^t\rbf \in f$ . By Corollary \ref{cor:deductionThm}, we conclude $\lb \pos\neg p_C^t\trir \bigvee_i(\neg p_{S_i})^t\rbf=S^{-1}(\overline{\Rhat(C)},\bigcup_i\overline{S_i})\in f$. 
\end{proof}
    
\begin{lemma}\label{fip}
    Let $\Frame F=\la W, R, \{S_x\}_{x\in W}\ra$ be a Veltman frame,  $f\in\textsf{U}(W)$, and $l$ be a proper filter over $W$. The set
    \(
    \{  A\subseteq W \mid \exists^{\mathrm{fin}}\{S_i\in l\}\Big(S^{-1}\big(\overline{A},\bigcup_i\overline{S_i}\big)\in f \Big)\}
    \)
    is closed under intersections.
\end{lemma}
\begin{proof}
     Assume $C,D\in \{  A\subseteq W \mid \exists^{\mathrm{fin}}\{S_i\in l\}\Big(S^{-1}\big(\overline{A},\bigcup_i\overline{S_i}\big)\in f \Big)\}$. 
     Then, for some sets $S^C\underset{\textit{fin}}{\subseteq}l$ and $S^D\underset{\textit{fin}}{\subseteq}l$, we have $S^{-1}(\overline{C},\bigcup_i \overline{S^C_i})\in f$ and $S^{-1}(\overline{D},\bigcup_j \overline{S^D_j})\in f$. 
     By Corollary \ref{theorem:AlgebraicVersionsOfSimpleILfacts}.\ref{item:AtrirBandBtoCThenAtrirC:theorem:AlgebraicVersionsOfSimpleILfacts}, $S^{-1}(X,Y)\subseteq S^{-1}(X,Y\cup Z)$. Since $f$ is closed under supersets, both $S^{-1}(\overline{C},\bigcup_{i,j} \overline{S^{C,D}_{i,j}})\in f$ and $S^{-1}(\overline{D}, \bigcup_{i,j} \overline{S^{C,D}_{i,j}})\in f$. 
     Define then ${\sf e}$ such that $\lb (p_C^t)\rbf=C$, $\lb p_D^t\rbf=D$, and $\lb p_{S_s}\rbf=S_k$ for each $S_k\in l$. By this definition of $\sf e$, both $\lb \neg p_C^t\trir \bigvee_{i,j} (\neg p_{S^{C,D}_{i,j}}^t)\rbf=S^{-1}(\overline{C},\bigcup_{i,j} \overline{S^{C,D}_{i,j}})$ and $\lb \neg p_D^t\trir \bigvee_{i,j} (\neg p_{S^{C,D}_{i,j}}^t)\rbf=S^{-1}(\overline{D},\bigcup_{i,j} \overline{S^{C,D}_{i,j}})$.
    
    Since $\il\vdash\Big(\neg p_C\trir \bigvee_{i,j} (\neg p_{S^{C,D}_{i,j}})\Big) \wedge \Big( \neg p_D\trir \bigvee_{i,j} (\neg p_{S^{C,D}_{i,j}})\Big)\rightarrow \Big((\neg p_C\vee \neg p_D)\trir \bigvee_{i,j} (\neg p_{S^{C,D}_{i,j}})\Big)$, by the deduction theorem $\{\neg p_C\trir \bigvee_{i,j} (\neg p_{S^{C,D}_{i,j}}),p_D\trir \bigvee_{i,j} (\neg p_{S^{C,D}_{i,j}})\}\vdash_\il (\neg p_C\vee \neg p_D){\trir} {\bigvee_{i,j}} (\neg p_{S^{C,D}_{i,j}})$. Then, by Lemma $\ref{lemma:implication}$, we then have $\lb (\neg p_C^t\vee \neg p_D^t)\trir \bigvee_{i,j} (\neg p_{S^{C,D}_{i,j}}^t)\rbf\in f$.
    Consequently, $\lb \neg(\neg p_C^t\vee\neg p_D^t)\rbf\in \{  A\subseteq W \mid \exists^{\mathrm{fin}}\{S_i\in l\}\Big(S^{-1}\big(\overline{A},\bigcup_i\overline{S_i}\big)\in f \Big)\}$; that is: $\lb p_C^t\wedge p^t_D\rbf=C\cap D\in f$, and we are done.    
\end{proof}

The above indeed suggests that each labelling lemma has its natural counterpart in the algebraic setting. Various observations from \cite{GorisBilkovaJoostenMikec:2020:ArXivLabels} are reflected in the following.
\begin{lemma}\ Let $\mathcal{F}=\langle W,R\{S_x\}_{x\in W} \rangle$ be a Veltman frame, and consider $f,g\in\textsf{U}(W)$ and $a,b\in\wp(W)$. The following hold:
    \begin{enumerate}
        \item \label{itemFirstBasic}
        $a\subseteq b$ and $f\prec_{b}g$, then $f\prec_{a}g$;
        \item \label{itemSecondBasic}
        $f\prec_a g$ and $R^{-1}h \subseteq g$, then  $f\prec_a h$.
        
\end{enumerate}
\end{lemma}

\begin{proof}
    To prove \ref{itemFirstBasic}, we assume that $a\subseteq b$ and $f\prec_b g$. Then, 
        \(
        \forall A\subseteq W\forall^{\mathrm{fin}}\{S_i\in b\}\Big(S^{-1}\big(\overline{A},\bigcup_i\overline{S_i}\big)\in f \Rightarrow A,\widehat{R^{-1}}A\in g\Big).
        \)
        We only need to observe that the finite choices of elements in $a$ are finite choices of elements in $b$ since $a\subseteq b$.

        To prove \ref{itemSecondBasic}, we assume that $f\prec_a g$ and $R^{-1}h\subseteq g$. Then, 
        \(\forall A\subseteq W\forall^{\mathrm{fin}}\{S_i\in b\}\Big(S^{-1}\big(\overline{A},\bigcup_i\overline{S_i}\big)\in f \Rightarrow A,\widehat{R^{-1}}A\in g\Big).\)
        Also, by Lemma \ref{lemma:interdef}, $g_{\nec}\subseteq h$. To prove that $f\prec_a h$, pick an arbitrary $A\subseteq W$ and an arbitrary finite choice of $S_i\in a$. Now, suppose that $S^{-1}\big(\overline{A},\bigcup_i\overline{S_i}\big)\in f$. Since $f\prec_a g$, we have that $A,\Rhat(A)\in g$. It suffices to show that $A,\Rhat(A)\in h$. On the one hand, since $g_{\nec}\subseteq h$ we derive that $A\in h$. On the other hand, by Lemma \ref{theorem:AlgebraicVersionsOfSimpleILfacts}.\ref{item:necAnecnecA:theorem:AlgebraicVersionsOfSimpleILfacts}, $\Rhat(A)\subseteq \Rhat(\Rhat(A))$ and since $\Rhat(A)\in g$ and $g$ is an ultrafilter, we know that $\Rhat(\Rhat(A))\in g$. Therefore, $\Rhat(A)\in g_{\nec}$ and, given that $g_{\nec}\subseteq h$, $\Rhat(A)\in h$.
\end{proof}
\begin{lemma}\label{labelextension} \ Let $\mathcal{F}=\langle W,R,\{S_w\}_{w\in W} \rangle$ be a Veltman frame, and consider $f,g\in\textsf{U}(W)$ and $a,b,y\in\wp(W)$. The following hold:
    \begin{enumerate}
        \item \label{labelextension1}
        if $f\prec_{a}g$, and if $x\subseteq y$ for some $x\in a$ then $f\prec_{a\,\cup\, \{y\}}g$;
        \item \label{labelextension2}
        if $f\prec_a g$, then $f\prec_{a \,\cup\,\Rhat(a) } g$, where $\Rhat(a)\coloneqq\{\Rhat(S)\mid S\in a\}$.
\end{enumerate}
\end{lemma}
\begin{proof}\
To prove item \ref{labelextension1} we pick an arbitrary $A\subseteq W$ and an arbitrary finite choice of $S_i\in a$. Now, assume that $S^{-1}\big(\overline{A},\bigcup_i\overline{S_i}\cup\overline{y}\big)\in f$. Since $x\subseteq y$, we know that $\overline{y}\subseteq \overline{x}$. By Corollary \ref{theorem:AlgebraicVersionsOfSimpleILfacts} we have that $S^{-1}\big(\overline{A},\bigcup_i\overline{S_i}\cup\overline{y}\big)\subseteq S^{-1}\big(\overline{A},\bigcup_i\overline{S_i}\cup\overline{x}\big)$ and, given that $f$ is an ultrafilter, we derive that $S^{-1}\big(\overline{A},\bigcup_i\overline{S_i}\cup\overline{x}\big)\in f$. Observe that, since $x\in a$, we have a finite choice of $S_j\in a$ consisting of the prior choice plus $x$ and we can write $S^{-1}\big(\overline{A},\bigcup_j\overline{S_j}\big)\in f$. Then, because $f\prec_a g$, we conclude that $A,\Rhat(A)\in g$ as required.\\

    To prove item \ref{labelextension2}, pick an arbitrary $A\subseteq W$ and an arbitrary finite choice of elements $S_i\in a$ and $\Rhat(S_j)\in \Rhat(a)$. Assume that $S^{-1}\big(\overline{A},\bigcup_i\overline{S_i}\cup\bigcup_j\overline{\Rhat(S_j)}\big)\in f$. By Lemma \ref{lemma:interdef}, \begin{equation}\label{eq1}
    S^{-1}\big(\overline{A},\bigcup_i\overline{S_i}\cup\bigcup_jR^{-1}(\overline{S_j})\big)\in f.
     \end{equation} By Theorem \ref{thm:provW}, since $\il\vdash \pos\alpha\trir\alpha$, we know that, for every $S_j$, $S^{-1}\big(R^{-1}(\overline{S_j}),\overline{S_j}\big)=W$ and, given that $f$ is an ultrafilter, $S^{-1}\big(R^{-1}(\overline{S_j}),\overline{S_j}\big)\in f$, for every $S_j$. Also, $\il\vdash \bigvee_i(\alpha_i\trir\beta_i)\to\bigvee_i \alpha_i\trir\bigvee_i \beta_i$, so by Corollary \ref{cor:deductionThm} we deduce that $\bigcup_jS^{-1}\big(R^{-1}(\overline{S_j}),\overline{S_j}\big)\subseteq S^{-1}\big(\bigcup_jR^{-1}(\overline{S_j}),\bigcup_j\overline{S_j}\big)$ which implies that $S^{-1}\big(\bigcup_jR^{-1}(\overline{S_j}),\bigcup_j\overline{S_j}\big)=W$ and, then, 
     \begin{equation}\label{eq2}
     S^{-1}\big(\bigcup_jR^{-1}(\overline{S_j}),\bigcup_j\overline{S_j}\big)\in f\end{equation} 
     Since $\il\vdash\alpha\trir (\beta\vee\gamma)\wedge(\gamma\trir\eta)\to \alpha\trir(\beta\vee\eta)$, by Corollary \ref{cor:deductionThm} we have that
     \[S^{-1}\big(\overline{A},\bigcup_i\overline{S_i}\cup\bigcup_jR^{-1}(\overline{S_j})\big)\cap S^{-1}\big(\bigcup_jR^{-1}(\overline{S_j}),\bigcup_j\overline{S_j}\big)\subseteq S^{-1}\big(\overline{A},\bigcup_i\overline{S_i}\cup\bigcup_j\overline{S_j}\big).\]
     By \ref{eq1} and \ref{eq2} and the fact that $f$ is an ultrafilter, we know that 
     $S^{-1}\big(\overline{A},\bigcup_i\overline{S_i}\cup\bigcup_j\overline{S_j}\big)\in f$.
     Finally, because $f\prec_ag$ we conclude that $A,\Rhat(A)\in g$, as we wanted to show.
\end{proof}

\begin{definition}\label{def:generatedFilter} Given a set $a\subseteq\powerset{W}$ with the finite intersection property, we define $l(a)$ to be the filter generated by $a$ as 
$l(a)\coloneqq \bigcap_{\substack{l' \text{ a filter over } W\\ a \subseteq l'}}l'.$
\end{definition}

The elements of the generated filter of some $a\subseteq\powerset{W}$ are characterized by the elements of $a$.

\begin{lemma}\label{lemma:generatedFilterProperFip}[From~\cite{BlackburnEtAll:2001:ModalLogic}]\label{lemma:charactGenFilter}
    Given some set $a\subseteq\powerset{W}$ we have that
    \[
    S\in l(a)\iff (\exists S_1,\ldots,S_n\in a\ S_1\cap\ldots\cap S_n\subseteq S).
    \]
    \end{lemma}
    Definition \ref{def:generatedFilter} can alos be applied to arbitrary $a\subseteq\powerset{W}$ and then the above lemma shows that $l(a)$ is a proper filter if and only if $a$ has the FIP.  


\begin{lemma}\label{lemma:filtergeneratedbyaFip}
    Let $\veltframe=\la W,R,\{S_w \mid w\in W\}\ra$ be a Veltman frame and consider $f,g\in\textsf{U}(W)$, $a\in\powerset{W}$ and $f\prec_a g$. Then, $f\prec_{l(a)} g$.
\end{lemma}
\begin{proof}
By Lemma~\ref{lemma:charactGenFilter}, for each $i\in\{1,\ldots,n\}$, there exists a finite subset $F_i\subseteq a$ such that $\bigcap F_i\subseteq S_i$. Equivalently, for each $i\in\{1,\ldots,n\}$, we have $\overline{S_i}\subseteq\bigcup_{S\in F_i}\overline{S}$. By letting $F:=\bigcup_{i=1}^{n}F_i$, we have a finite subset $F\subseteq a$ such that $\bigcup_{i=1}^n\overline{S_i}\subseteq\bigcup_{S\in F}\overline{S}$. Now, consider some $S^{-1}\big(\overline{A},\bigcup_{i=1}^n\overline{S_i}\big)\in f$. Because $\il \vdash (\beta\to \gamma)\to (\alpha\trir \beta\to \alpha\trir \gamma)$, by Corollary~\ref{cor:deductionThm} we have that $S^{-1}\big(\overline{A},\bigcup_{i=1}^n\overline{S_i}\big)\subseteq S^{-1}\big(\overline{A},\bigcup_{S\in F}\overline{S}\big)$\footnote{We note that given $A,B\in\wp(W)$, then $ A\subseteq B\iff \overline{A}\cup B= W$}. Since $f$ is a filter, and therefore upward closed, we derive that $S^{-1}\big(\overline{A},\bigcup_{S\in F}\overline{S}\big)\in f$. Finally, since $F\subseteq a$ is finite and $f\prec_a g$, we have that $A,\Rhat(A)\in g$, and we are done.
\end{proof}

\section{Elementary Equivalence and Modal Saturation}\label{section:mainResults}

The following two lemmas ensure that we can always find adequate ultrafilters.

\begin{lemma}\label{triralgebra}
    Let $\Frame F=\la W,R,\{S_x\}_{x\in W}\ra$ a Veltman frame, and $f\prec_l g$, for some set $l$ over $W$. If $S^{-1}(A,B) \in f$ and $A\in g$, then there is some $h$ with $f\prec_l h$ and $B\in h$.
\end{lemma}

\begin{proof}

    Assume that $S^{-1}(A,B)\in f$, that $f\prec_l g$ and $A\in g$. Seeking a contradiction, suppose that, for every $h$, if $f\prec_l h$, then $B\notin h$. Then, 
    $\{B\}\cup\{C_i,\Rhat(C_i)\mid \exists^{\sf fin} S_i\in l\ S^{-1}(\overline{C_i},\bigcup_i \overline{S_i})\in f\}$
    does not have the FIP: by Lemma \ref{fip}, there is $C\in \{C_i,\Rhat(C_i)\mid \exists^{\sf fin} S_i\in l\ S^{-1}(\overline{C_i},\bigcup_i \overline{S_i})\in f\}$
    such that $B\cap C\cap\Rhat(C)=\varnothing$. Now, let $\sf e$ be a valuation such that $\lb p_A^t\rbf=A$, $\lb p_B^t\rbf=B$, $\lb p_C^t\rbf=C$, and, for every $S_i$, $A_{p_{S_i}}=S_i$. Therefore, we have $\lb p^t_B\wedge p^t_C\wedge \nec p^t_C\rbf=\varnothing$, that is: $\lb\neg(p_B^t\wedge p_C^t\wedge\nec p_C^t)\rbf=W\in f'$, for every $f'\in\textsf{U}(W)$. Moreover, from $\neg(p_B\wedge p_C\wedge\nec p_C)\vdash_{\il} p_B\to (\neg p_C\vee\pos \neg p_C)$ and Lemma \ref{lemma:implication}, we have $\lb p^t_B\to (\neg p^t_C\vee\pos \neg p^t_C)\rbf\in f'$, for all ultrafilters $f'\in\textsf{U}(W)$. By Lemma \ref{infforallevaluationsW}, we have  $\lb \nec\big(p^t_B\to (\neg p^t_C\vee\pos \neg p^t_C)\big)\rbf=W$, and, by $\nec\big(p_B\to (\neg p_C\vee\pos \neg p_C)\big)\vdash_{\il}p_B\trir (\neg p_C\vee\pos \neg p_C)\vdash_{\il} p_B\trir\neg p_C$, Lemma \ref{lemma:implication}, and Lemma \ref{infforallevaluationsW}, we have that $\lb p^t_B\trir\neg p_C^t\rbf=W\in f$.
    
    By assumption, we know that there are some $S_i\in l$ such that $S^{-1}(\overline{C},\bigcup_i\overline{S_i})=\lb \neg p^t_C\trir \bigvee_i\neg p^t_{S_i}\rbf \in f$ and that $S^{-1}\big(\overline{\overline{A}},B\big)=\lb \neg\neg p^t_A\trir p^t_B\rbf\in f$. Given that $\neg\neg p_A\trir p_B,\ p_B\trir \neg p_C,\ \neg p_C\trir \bigvee_i\neg p_{S_i}\vdash_{\il} \neg\neg p_A\trir \bigvee_i\neg p_{S_i}$, by Lemma \ref{lemma:implication}, we derive that $S^{-1}\big(\overline{\overline{A}},\bigcup_i\overline{S_i}\big)=\lb \neg\neg p^t_A\trir \bigvee_i\neg p_{S_i}^t\rbf\in f$ and, since $f\prec_l g$, we conclude that $\overline{A}\in g$, which contradicts our hypothesis.
\end{proof}




\begin{lemma}\label{notriralgebra}
    Given $\Frame F=\langle W,R,\{S_w \mid w\in W\}\rangle$ a Veltman frame, $f\in\textsf{U}(W)$, and $\overline{S^{-1}(A,B)} \in f$, there is some ultrafilter $g$ and some proper filter $l$ over $W$ with $A \in g$ and $\overline B\in l$ and $f\prec_l g$.
\end{lemma}
\begin{proof}\
Assume $\overline{S^{-1}(A,B)} \in f$ and suppose for a contradiction that for all ultrafilters $g$ over $W$ and for all proper filters $l$, if $f\prec_l g$, then $A\notin g$. Let $l=\{S:\overline B\subseteq S\}$\footnote{Clearly $\overline{B}\in l$ as $\overline B\subseteq\overline{B}$}. 
Then, we have $\{A\}\cup\{C_i,\Rhat(C_i)\mid \exists^{\sf fin}\ \{ S_i \in  l\}\ \big(S^{-1}(\overline{C_i}, \bigcup_i S_i)\big)\}$ does not have the finite intersection property. By Lemma $\ref{fip}$, there is some $C\in \{C_i,\Rhat(C_i)\mid \exists^{\textsf{fin}} S_i\ S^{-1}(\overline{C_i}, B)\}$ we have that $C\cap \Rhat(C) \cap A=\varnothing$.

Let $\sf e$ be a valuation such that $\lb p_A\rbf=A$, $\lb p_C\rbf=C$ and $\lb p_{S_i}\rbf=S_i$ for each $S_j\in l$.
We observe that $\lb p^t_A\wedge p^t_C\wedge \nec p^t_C \rbf=\varnothing$. But then $\lb\neg(p^t_A\wedge p^t_C\wedge \nec p^t_C)\rbf=W \in f$. We observe that $\neg(p_A\wedge p_C\wedge\nec p_C) \vdash_{\il}  p_A\rightarrow \neg p_C\vee \pos \neg p_C$ and that by Lemma \ref{lemma:implication} we obtain that $\lb p_A^t\rightarrow \neg p_C^t\vee \pos \neg p_C^t\rbf \in f$. By Lemma \ref{infforallevaluationsW} we obtain that $\Rhat(\lb p_A^t\rightarrow \neg p_C^t\vee \pos \neg p_C^t\rbf)=W \in f$. We observe that $\nec (p_A\rightarrow \neg p_C\vee \pos \neg p_C)\vdash_{\il}  p_A\trir \neg p_C\vee\pos\neg p_C$ by $\play{1}$ and, again by \ref{lemma:implication} $\lb p_A^t\trir \neg p_C^t\vee\pos\neg p_C^t\rbf \in f$. We also observe that $\il\vdash \neg p_C\trir \neg p_C$ and $\il\vdash \pos\neg p_C\trir \neg p_C$. Whence by using axioms $\play{3}$ and $\play{2}$ and Lemma \ref{lemma:implication} we resolve that $\lb p^t_A\trir \neg p^t_C\rbf=S^{-1}(A,\overline{C})\in f$.
But then, we observe that, from the fact that for a finite choice of $S_i\in l$ we have $S^{-1}(\overline{C},\bigcup_i \overline{S_i})\in f$, we also have $\lb \neg p_C^t\trir p^t_B\rbf=S^{-1}(\overline{C},B)\in f$, by Corollary \ref{theorem:AlgebraicVersionsOfSimpleILfacts}.\ref{item:AtrirBandBtoCThenAtrirC:theorem:AlgebraicVersionsOfSimpleILfacts}. It can be proven that $p_A\trir \neg p_C ,\neg p_C\trir p_B \vdash_{\il} p_A\trir p_B$. By Lemma \ref{lemma:implication} we have $\lb p_A^t\trir p_B^t\rbf=S^{-1}(A,B)\in f$ which contradicts our hypothesis. 
\end{proof}


\begin{lemma}\label{extensioniffTranslation}
    Given $\Frame M=\la\Frame F, \eval\ra$ a Veltman model where $\Frame F=\langle W,R,\{S_w \mid w\in W\}\rangle$, let $f\in\textsf{U}(W)$ and $\sf e$ be a  valuation such that $\lb A_p\rbf=\lb p\rb_{\Frame M}$ for every $p\in \atom$. Then, for any $\varphi\in\form$, we have $\lb \varphi^t\rbf\in f\iff \lb \varphi\rb_{\Frame M}\in f$.
\end{lemma}
\begin{proof}
    By a standard structural induction on the complexity of $\varphi$.
    \begin{enumerate}
    \item $\varphi=\bot$; then $\lb \bot^t\rbf=\varnothing=\lb\bot\rb_{\Frame M}$
    \item 
    $\varphi=p$ with $p\in\atom$; By definition $\lb p^t\rbf=\lb p\rb_{\Frame M}$.
    \item
    $\varphi=\psi\rightarrow\gamma$: By definition $\lb \psi^t\rightarrow\gamma^t\rbf= \lb\overline{\psi^t}\rbf\cup\lb \gamma^t\rbf$. By induction hypothesis, we have that $\lb \psi^t\rbf=\lb \psi\rb_{\Frame M}$, so that $\lb\overline{\psi^t}\rbf\cup\lb\gamma^t\rbf=\lb\overline{\psi}\rb_{\Frame M}\cup \lb\gamma\rb_{\Frame M}=\lb \psi\rightarrow\gamma\rb_{\Frame M}$, as required.
    \item
    $\varphi=\nec \psi$: 
    By induction hypothesis, we have $$\lb \nec \psi^t\rbf=\Rhat(\lb\psi^t\rbf)= \Rhat(\lb\psi\rb_{\Frame M})\in f\iff \lb \nec \psi\rb_{\Frame M}\in f,$$ as required.
    \item 
    $\varphi=\psi\trir\gamma$: 
    By induction hypothesis, we have that $$\lb \psi^t\trir\gamma^t\rbf=S^{-1}(\lb\psi^t\rbf,\lb\gamma^t\rbf)=  S^{-1}(\lb \psi\rb_{\Frame M},\lb \gamma\rb_{\Frame M})\in f\iff \lb \psi\trir\gamma\rb_{\Frame M}\in f,$$ as required.   
    \end{enumerate}
\end{proof}

\begin{lemma}\label{oneDirectionTruthLemmmaUltrafilterExtension}
    Given $\Frame M=\la\Frame F, \eval\ra$ a Veltman model where $\Frame F=\langle W,R,\{S_w \mid w\in W\}\rangle$. For any formula $\varphi$ and any $\la \Pi_x, \sigma \ra$ in the ultrafilter extension we have     \[
    \Frame M , x \Vdash \varphi \ \ \Longrightarrow \ \ \ultraFilterCounterpart{\Frame M}, \la \Pi_x, \sigma \ra \Vdash \varphi .
    \]
\end{lemma}
\begin{proof}
    By induction on the complexity of the formula $\psi$ we show that $\semanticExtensionInAModelThreeArguments{\psi}{\Frame M}{\ } \in f\Longrightarrow \ultraFilterCounterpart{\Frame M}, \la f, \sigma \ra \Vdash \psi$ .

If $\varphi=p$ for some $p\in \atom$ follows from the definition. 
The Boolean connectives follow directly from the inductive hypothesis, although, for clarity, we will show the case for negation:
Let $\varphi= \neg \psi$ Assume $\semanticExtensionInAModelThreeArguments{\neg\psi}{\Frame M}{}\in f$. Assume for a contradiction that for some $\la f,\sigma\ra\in \ultraFilterCounterpart{W}, \ultraFilterCounterpart{M}, \la f,\sigma\ra\not\Vdash \neg\psi$. Then $\ultraFilterCounterpart{M}, \la f,\sigma\ra\Vdash \psi$. By induction hypothesis (in contraposition) this entails that $\overline{\semanticExtensionInAModelThreeArguments{\psi}{\Frame M}{}}\notin f$. But then $\semanticExtensionInAModelThreeArguments{\neg \psi}{\Frame M}{}\notin f$ which contradicts our initial assumption.

Since the $\nec$ and $\pos$ modalities can be expressed in terms of $\trir$ we focus on the latter case: let $\psi=\alpha\trir\beta$.

Assume $\semanticExtensionInAModelThreeArguments{\alpha\trir\beta}{\Frame M}{\ } \in f$ and $\sigma$ is a sequence of proper filters over $W$ in such a way that $\la f,\sigma\ra\in \ultraFilterCounterpart{W}$. Assume there is some $\la g,\tau\ra\in\ultraFilterCounterpart{W}$ such that $\la f,\sigma\ra \ultraFilterCounterpart{R} \la g,\tau\ra$ and $\ultraFilterCounterpart{\Frame M},\la g,\tau\ra\Vdash \alpha$. We define $\sf e$ so that, for all $p\in \atom$, we have $\lb A_p\rbf= \semanticExtensionInAModelThreeArguments{p}{\Frame M}{\ }$. We will denote $A= \lb \alpha^t\rbf$ and $B=\lb \beta\rbf$.
By Lemma \ref{extensioniffTranslation} we have $A=\semanticExtensionInAModelThreeArguments{\alpha}{\Frame M}{\ }$ and $B=\semanticExtensionInAModelThreeArguments{\beta}{\Frame M}{\ }$ and that $S^{-1}(A,B)\in f$, $\la f,\sigma\ra \ultraFilterCounterpart{R} \la g,\tau\ra$ and $A\in g$. We observe that $\tau = \sigma\smallfrown\la l\ra\smallfrown \overline \tau$ for some $l$ and $\overline \tau$. Then, by Lemma \ref{RlTransitivity}, we have that $f\prec_l g$ and, by Lemma \ref{triralgebra}, we have that there is some $h\in \wp(W)$ such that $f\prec_l h$ and $B\in h$, so that that $\la f,\sigma\ra \ultraFilterCounterpart{R}\la h, \la \sigma\smallfrown \la l\ra\ra\ra$ and $(\tau)_{|\sigma|}=(\sigma\smallfrown \la l\ra)_{|\sigma|}$. Therefore $\la g,\tau\ra \ultraFilterCounterpart{S}_{\la f,\sigma\ra} \la h, \la \sigma\smallfrown \la l\ra\ra\ra$. Since $\la g,\tau\ra$ was arbitrary, we conclude that $\ultraFilterCounterpart{\Frame M}, f\Vdash \alpha\trir \beta$.
\end{proof}

\begin{theorem}\label{truthLemmmaUltrafilterExtension}
    Given $\Frame M=\la\Frame F, \eval\ra$ a Veltman model where $\Frame F=\langle W,R,\{S_w \mid w\in W\}\rangle$. For any formula $\varphi$ and any $\la \Pi_x, \sigma \ra$ in the ultrafilter extension we have     \[
    \Frame M , x \Vdash \varphi \ \ \Longleftrightarrow \ \ \ultraFilterCounterpart{\Frame M}, \la \Pi_x, \sigma \ra \Vdash \varphi .
    \]
\end{theorem}
\begin{proof}
We will show that 
\begin{equation*}
\semanticExtensionInAModelThreeArguments{\psi}{\Frame M}{\ } \in f \ \iff \ 
    \ultraFilterCounterpart{\Frame M}, \la f, \sigma \ra \Vdash \psi
\end{equation*}

by induction on the complexity of the formula $\psi$

If $\varphi=p$ for some $p\in \atom$ it follows from the definition. For Boolean connectives it follows directly from the induction hypothesis and Lemma \ref{oneDirectionTruthLemmmaUltrafilterExtension}. As stated in the proof of the previous Lemma, $\nec$ and $\pos$ modalities can be expressed in terms of $\trir$, so we will only provide proof for the latter case. Let $\psi=\alpha\trir\beta$; the left-to-right direction is proven in Lemma \ref{oneDirectionTruthLemmmaUltrafilterExtension}.

For the Right-to-left direction assume that $ \semanticExtensionInAModelThreeArguments{\alpha\trir\beta}{\Frame M}{\ } \notin f$. Since $f$ is an ultrafilter, we have that $\overline{\semanticExtensionInAModelThreeArguments{\alpha\trir\beta}{\Frame M}{\ }}\in f$ and, 
per definition, we obtain that $\semanticExtensionInAModelThreeArguments{\neg(\alpha\trir\beta)}{\Frame M}{\ }\in f$. Then, by Lemma \ref{oneDirectionTruthLemmmaUltrafilterExtension}, 
$\ultraFilterCounterpart{\Frame M}, \la f, \sigma \ra \Vdash \neg(\alpha\trir\beta)$ and hence $\ultraFilterCounterpart{\Frame M}, \la f, \sigma \ra \not\Vdash (\alpha\trir\beta)$ as intended.

\end{proof}
We also have that the ultrafilter extension is modally saturated.

\begin{definition}[Modal saturation]\label{def:mSaturated}
    Let $\mathcal{M}$ be a modal model and $x\in\mathcal{M}$. We call a set of $\Sigma$ of modal sentences \textbf{locally possible} in $x$ whenever for any finite $\Sigma'\subseteq\Sigma$ we have $\mathcal{M},x\Vdash\pos\bigwedge \Sigma'$. The set $\Sigma$ is called \textbf{possible} (outright) at $x$ whenever there is some $y\in\mathcal{M}$ with $xRy$ and $\mathcal{M},y\Vdash\Sigma$.

    We call a modal model $\mathcal{M}$ \textbf{modally saturated} whenever for any $x\in\mathcal{M}$ we have that any set that is locally possible at $x$ is possible at $x$.
\end{definition}

\begin{theorem}
    For any Veltman model $\Frame M=\la\Frame F, \eval \ra$, the ultrafilter extension $\ultraFilterCounterpart{\Frame M}$ is modally saturated.
\end{theorem}
\begin{proof}
    Consider an arbitrary $\la f,\sigma\ra\in\ultraFilterCounterpart{\Frame M}$ and let $\Sigma$ be locally possible at $\la f,\sigma\ra$ in $\ultraFilterCounterpart{\Frame M}$. We have to show that $\Sigma$ is possible. Let $\lb\Sigma\rb_{\Frame M}\coloneqq \{\lb\varphi\rb_{\Frame M} \mid \varphi\in\Sigma\}$. By combining Lemma \ref{lemma:interdef}, Lemma \ref{lemma:generatedFilterProperFip}, Lemma \ref{lemma:ultrafilterPrinciple} and Theorem \ref{truthLemmmaUltrafilterExtension} it is enough to see that $f_{\nec}\cup \lb\Sigma\rb_{\mathcal{M}}$
    has the finite intersection property. Since any finite part $T$ of $\Sigma$ has a $\ultraFilterCounterpart{R}$-successor $\la g,\sigma'\ra $ with $\ultraFilterCounterpart{\Frame M}, \la g,\sigma'\ra \Vdash \bigwedge T$, by Theorem \ref{truthLemmmaUltrafilterExtension} we have $\bigcap_{\varphi\in T}\lb\varphi \rb_{\Frame M}\in g$.  By Lemma \ref{RlTransitivity}, Lemma \ref{lemma:assuringThenSuccessors} and Lemma \ref{lemma:interdef}, we see that $f_{\nec}\subseteq g$.  Therefore, for any $A\in f_{\nec}$, we have $\varnothing\neq A\cap\bigcap_{\varphi\in T}\lb\varphi\rb\in g$. Given that $A$ and $T$ were arbitrary, we conclude that $f_{\nec}\cup\lb\Sigma\rb_{\Frame M}$ has the finite intersection property.
\end{proof}
We also have a certain saturation at the label level.
\begin{theorem}
Let $f\in\textsf{U}(W)$. If, for every finite $l'\subset l$, there exists a $g$ such that $f\prec_{l'}g$, then there exists an $h$ such that $f\prec_l h$.
\end{theorem}
\begin{proof}
    Suppose not. Then, for every $h\in\textsf{U}(W)$, there is some $A\subseteq W$ and some finite choice of $S_i\in l$ such that $S^{-1}\big(\overline{A},\bigcup_i\overline{S_i}\big)\in f$ and $A\notin h$ or $\Rhat(A)\notin h$. Observe that the finite choice of $S_i\in l$ is a finite choice of $S_i\in l'$, for some finite $l'\subseteq l$. By assumption, there is some ultrafilter $g\in\textsf{U}(W)$ such that $f\prec_{l'}g$. Since $S^{-1}\big(\overline{A},\bigcup_i\overline{S_i}\big)\in f$, we have that $A,\Rhat(A)\in g$, which contradicts our hypothesis. 
\end{proof}

\section{Conclusions}

Through our example of a modally undefinable first-order frame condition we motivated the need of a Goldblatt-Thomason theorem~\cite{goldblatt-thomason} for interpretability logic. We proceeded to give a definition for ultrafilter extensions of Veltman frames and models and proved the main properties of these. The definition that we present resembles the canonical model construction~\cite{JonghVeltman:1990:ProvabilityLogicsForRelativeInterpretability,GorisBilkovaJoostenMikec:2022:JournalLabels} which came slightly unexpected since the ultrafilter extension starts from a concrete model. For this reason, and in general, it would be interesting to see how ultrafilter extensions fare under Verbrugge semantics~\cite{JoostenMasMikecVukovic:2024:OverviewVerbrugge}.

\section*{Acknowledgement} The second author has received financial support through the projects PID2023-151396OB-I00 and PID2023-149556NB-I00 of the Spanish Ministry of Science and Innovation, through 2021 SGR 00348 of the \textit{Generalitat de Catalunya} and, through the Severo Ochoa and María de Maeztu Program for Centers and Units of Excellence in R\&D (CEX2020-001084-M), the Spanish State Research Agency. We thank CERCA Programme/Generalitat de Catalunya for institutional support.

The fourth author has received financial support through the project project SERICS – Security and Rights in the CyberSpace PE0000014 within PNRR, M4C2 I.1.3, funded by the European Union - NextGenerationEU (MUR Code: 2022CY2J5S). We thank the Istituto Nazionale di Alta Matematica/INdAM group GNSAGA for institutional support.

We are very grateful for two very detailed referee reports. They greatly improved the presentation of this paper. Refereeing is an anonymous task, which is not well accounted for in the publish-or-perish culture bestowed upon us by various enthusiasts who hold a master in business administration, certainly not the \textit{quality} of refereeing. In light of this, we are therefore extremely grateful for the tremendous effort the two referees made to help improve this paper.

\end{document}